\documentclass[preprint]{imsart}

\RequirePackage[OT1]{fontenc}
\RequirePackage{amsthm, amssymb, amsmath, verbatim, comment, color}
\RequirePackage[numbers]{natbib}
\RequirePackage[colorlinks,citecolor=blue,urlcolor=blue]{hyperref}

\startlocaldefs

\newcommand{\OB}{Ob{\l}{\'o}j}

\newtheorem{theorem}{Theorem}

\newtheorem{corollary}[theorem]{Corollary}
\newtheorem{definition}[theorem]{Definition}
\newtheorem{lemma}[theorem]{Lemma}
\newtheorem{proposition}[theorem]{Proposition}

\theoremstyle{remark}
\newtheorem{remark}[theorem]{Remark}
\newtheorem{example}[theorem]{Example}

\newcommand{\E}{\mathbb{E}}
\renewcommand{\P}{\mathbb{P}}

\newcommand{\R}{\mathbb{R}}

\newcommand{\be}{\begin{equation}}
\newcommand{\ee}{\end{equation}}

\DeclareMathOperator{\supp}{supp}
\newcommand{\bea}{\begin{eqnarray}}
\newcommand{\bes}{\begin{subequations}}
\newcommand{\ees}{\end{subequations}}
\newcommand{\bgt}{\begin{gather}}
\newcommand{\egt}{\begin{gather}}
\newcommand{\eea}{\end{eqnarray}}
\newcommand{\beaa}{\begin{eqnarray*}}
\newcommand{\eeaa}{\end{eqnarray*}}

\usepackage{bbm}

\renewcommand{\epsilon}{\varepsilon}

\newcommand{\fourIdx}[5]{%
\setbox1=\hbox{\ensuremath{^{#1}}}%
 \setbox2=\hbox{\ensuremath{_{#2}}}%
 \setbox5=\hbox{\ensuremath{#5}}%
 \hspace{\ifnum\wd1>\wd2\wd1\else\wd2\fi}%
 \ensuremath{\copy5^{\hspace{-\wd1}\hspace{-\wd5}#1\hspace{\wd5}#3}%
 _{\hspace{-\wd2}\hspace{-\wd5}#2\hspace{\wd5}#4}%
 }}

\numberwithin{equation}{section}
\numberwithin{theorem}{section}

\renewcommand{\subset}{\subseteq}

\usepackage{subcaption}
\usepackage{asymptote}
\usepackage{graphicx}

\numberwithin{equation}{section}
\theoremstyle{plain}

\endlocaldefs
\begin{document}

\begin{frontmatter}
\title{Optimal martingale transport between radially symmetric
marginals in general dimensions
}
\runtitle{Martingales between radially symmetric
marginals}

\begin{aug}
\author{\fnms{Tongseok} \snm{Lim}\thanksref{t1}
\ead[label=e1]{dongseok0213@gmail.com}}
\ead[label=e2]{lds@math.ubc.ca}

\thankstext{t1}{The author gratefully acknowledges support from a doctoral graduate fellowship from the University of British Columbia, and from the Austrian Science Foundation (FWF) through grant Y782. \copyright 2016 by the author.
}
\runauthor{Tongseok Lim} 

\affiliation{Vienna University of Technology}
\address{Mathematical Stochastics, TU Vienna\\
1040 Wien, Austria\\
\printead{e1}}

\end{aug}

\begin{abstract}
We determine the optimal structure of couplings for the \emph{Martingale transport problem} between radially symmetric initial and terminal laws $\mu, \nu$ on $\R^d$ and show the uniqueness of optimizer. Here optimality means that such solutions will minimize the functional $\E f(||X-Y||)$ where $f$ is concave and strictly increasing, and the dimension $d$ is arbitrary. 
\end{abstract}



\end{frontmatter}

\section{Optimal transport problem and its variant} 
\subsection{Optimal transport problem}
This paper focuses on the structure of probability measures which solve certain optimization problems. The prototype is the optimal mass transport problem: for a given cost function $c:\R^d\times \R^d \to \R$ and two Borel probability measures $\mu, \nu$ on $\R^d$, we consider:
\begin{equation}\label{TP}
\mbox{Minimize} \,\,\, \text{cost}[\pi] = \int_{\R^d\times \R^d} c(x,y) \,d\pi(x,y)
\end{equation}
over all $\pi\in \Pi(\mu,\nu)$, where $\Pi(\mu,\nu)$ is the set of {\it mass Transport Plans}, or {\it couplings}, i.e. the set of probabilities $\pi$ on $\R^d \times \R^d$ with marginals $\mu$ and $\nu$ on $\R^d$. We interpret the transport plan $\pi$ as follows: for $A,B \subset \R^d$, $\pi(A \times B)$ is the amount of mass transported by the plan $\pi$ from the resource domain $A$ to the target range $B$.  
An equivalent probabilistic formulation is to consider the following problem:
\begin{align}\label{tp}
\mbox{Minimize}\quad  
\E_{\rm P} \,c(X,Y)
\end{align}
over all joint random variables $(X,Y): \Omega \to \R^d \times \R^d$ with given laws $X \sim \mu$ and $Y \sim \nu$ respectively. \\

In 1781, Gaspard Monge \cite{Mo1781}  formulated the following question that was relevant to his work in engineering: Given
two sets $U, V$ in $\R^d$ of equal volume, find the optimal volume-preserving map between
them, where optimality is measured against the cost function $c(x,y)$ of transporting particle $x$ to $y$. The optimal map should then minimize the total
cost of redistributing the mass of $U$ through $V$. Much later,  Kantorovich generalized the Monge problem and proposed the above formulation.


In Monge's original problem \cite{Mo1781}, the cost was simply the Euclidean distance $c(x,y) = |x-y|$. Even for this seemingly simple case, it took two centuries before Sudakov \cite{sudakov}, Evans \cite{evans}, Gangbo-McCann \cite{gm}, Ambrosio-Kirchheim-Pratelli \cite{akp, ap}, Caffarelli-Feldman-McCann \cite{cfm}, Bianchini-Cavalletti \cite{bc}, Ma-Trudinger-Wang \cite{MTW, TW1, TW2} and others showed rigorously that an optimal transport map exists. For general account of the theory, see Villani \cite{Vi03, Vi09}.

More recently, a new direction emerged where  the transport plans are assumed to be martingales.  In the sequel, we shall describe the problem, its motivation, and our contributions.

\subsection{Martingale optimal transport problem}
Now we consider the following problem
\begin{equation}\label{MGTP}
\mbox{Minimize}\quad \text{Cost}[\pi] = \int_{\R^d \times \R^d} c(x,y)\,d\pi(x,y)\quad\mbox{over}\quad \pi\in {\rm MT}(\mu,\nu)
\end{equation}
where {\rm MT}$(\mu,\nu)$ (Martingale Transport plan) is the set of joint probabilities on $\R^d \times \R^d$ having $\mu$ and $\nu$ as its marginals, such that for $\pi\in$ {\rm MT}$(\mu,\nu)$, its disintegration $\pi_x$ has its barycenter at $x$. In other words, for any convex function $\xi$ on $\R^d$, disintegration measure $(\pi_x)_x$ with respect to $\mu$ must satisfy
\begin{align}\label{constraint}
\xi(x) \le \int \xi(y)\,d\pi_x (y) \quad  \text{for} \quad \mu-a.e. \, x.
\end{align} 
We interpret  disintegration as conditional probability 
$$d\pi_x (y) = \P(Y=y|X=x).$$

Probabilistic description of the problem is the following: we study 
\begin{align}\label{opt}
\text{Minimize}\quad  
\E_{\rm P} \,c(X,Y)
\end{align}
over all martingales $(X,Y)$ on a probability space $(\Omega, {\mathcal F}, P)$ into $\R^d \times \R^d$ (i.e. $E[Y|X]=X$) with prescribed laws $X \sim \mu$ and $Y \sim \nu$.\\

It is shown in \cite{St65} that {\rm MT}$(\mu, \nu)$ is nonempty if and only if $\mu$ and $\nu$ are in convex order.
\begin{definition}\label{ConvexOrderDef}
Measures $\mu$ and $\nu$ are said to be in convex order if
\begin{enumerate}
\item they have finite mass and finite first moments,
\item for convex functions $\xi$ defined on $\R^d$, $\int \xi\,d\mu\leq \int \xi\, d\nu.$
\end{enumerate}
In that case we will write $\mu\leq_c\nu$. 
\end{definition}

Note that measures $\mu, \nu$ in $\R^1$ having the same finite mass and the same first moments are in convex order if and only if $ \int (x-k)_+ \, d \mu(x)\leq \int (x-k)_+\, d \nu(x)$ for all real $k$. Also note that with this notation, \eqref{constraint} can be written as $\delta_x \le_c \pi_x$. \\


  Along with D. Hobson's pioneering observation of the importance of Skorokhod embedding techniques  in the ``model-free" approach to finance and asset pricing \cite{Ho11}, much related research  has been done in the context of Skorokhod embedding and Martingale optimal transport; e.g.  Beiglb{\"o}ck-Henry-Labordere-Penkner \cite{BeHePe11}, Beiglb{\"o}ck-Henry-Labordere-Touzi \cite{BeHeTo}, Beiglb{\"o}ck-Juillet \cite{bj}, Beiglb{\"o}ck-Nutz-Touzi \cite{bnt}, Hobson-Klimmek \cite{HoKl12}, Hobson-Neuberger \cite{HoNe11} for discrete time case and Beiglb{\"o}ck-Cox-Huesmann \cite{bch}, Dolinsky-Soner \cite{ds1}, Galichon-Henry-Labordere-Touzi \cite{GaHeTo11}, Guo-Tan-Touzi \cite{GTT} for continuous time case.  By no means this  list completely represents the rapidly expanding subject. For the link between MOT, SEP and model-free approach to finance, we refer the reader to Henry-Labordere \cite{HL}, Hobson \cite{Ho98}, Ob{\l}{\'o}j \cite{Obloj}.

We note that the above cited papers are all concerned with dimension one. In this paper, we show that the optimal martingale problem has a unique solution in case the marginals $\mu$, $\nu$ are radially symmetric in arbitrary dimensions, which is satisfied by important distributions such as Gaussians. To the best of the author's knowledge, this is the first such result to be established in arbitrarily high dimensions, along with a companion paper \cite{GKL2} which deals with the
general marginals case. In view of the fact that the optimal transport theory \eqref{TP} in higher dimension has had a profound influence on many areas of mathematics, physics and economics, we hope that the theory of martingale optimal transport in higher dimension will also find many important applications.

In this paper, we will focus on the cost function (note $x,y \in \R^d$)
\begin{align}\label{cost}
c(x,y) = f(|x-y|), \text{ where } f(0)=0, f' >0, f'' \le 0 \text{ on } \R_+.
\end{align}
That is, $f : \R_+ \to \R_+$ is strictly increasing and concave. The power cost $c(x,y) = |x-y|^p$, $0<p \le 1$, is a particular example.  We note that as our cost function is lower-semicontinuous and nonnegative, solutions to \eqref{MGTP} exist; see, e.g. \cite[Theorem 1]{BeHePe11} (We note that \cite[Theorem 1]{BeHePe11} is stated in one-dimensional marginals case. But the existence of solutions to \eqref{MGTP} follows from the standard argument, i.e. from the compactness and nonemptiness of ${\rm MT}(\mu,\nu)$ with the lower-semicontinuity of the cost functional on it. For this, the proof in \cite[Proposition 2.4]{BeHePe11} works just as well for $d$-dimensional setting).

Now we introduce the main theorem.

\begin{theorem}\label{mainresult} Suppose that $\mu, \nu$ are radially symmetric probability measures on $\R^d$ which are in convex order and  $(\mu-\mu\wedge \nu)(\{0\} )= 0$. Assume that either $\mu$ is absolutely continuous, or that there exists a ball $B_r$  (with center $0$ and radius $r$, either open or closed) such that $\mu-\mu\wedge \nu$ is concentrated on $B_r$  while $\nu-\mu\wedge \nu$ is concentrated on $\R^d \setminus B_r$. Then there is a unique minimizer $\pi$ for the problem \eqref{MGTP} with respect to the cost \eqref{cost}, and for $\mu$ almost every $x$,  disintegration $\pi_x$ is concentrated on the one-dimensional subspace $L_x = \{ax \,|\, a \in \R\}$. Furthermore, if $\mu$ is absolutely continuous with respect to Lebesgue measure and $\mu \wedge \nu =0$, then $\pi_x$ is supported at two points on $L_x$.
\end{theorem}

We note that \cite{GKL2} studied the optimal martingale transport problem in general dimensions as well, and they conjectured the following extremal property of minimizers.\\

\noindent {\bf Conjecture:}  {\it  Consider the cost function 
$c(x, y) =|x-y|$  and assume that $\mu$ is absolutely continuous with respect to Lebesgue measure on $\R^d$, and that $\mu \wedge \nu =0$. If $\pi$ is a martingale transport that minimizes $(\ref{MGTP})$, then for $\mu$ almost every x, the support of disintegration $\supp \pi_x$ consists of  $k+1$ points that form the vertices of a $k$-dimensional polytope,  where $k:=k(x)$ is the dimension of the linear span of $\supp \pi_x$. Finally, the minimizing solution is unique.}\\

Therefore, Theorem \ref{mainresult} can be seen as an affirmative answer for the above conjecture when the marginals $\mu$ and $\nu$ are radially symmetric on $\R^d$, and in this case $k(x) \equiv 1$. On the other hand, \cite{GKL2} showed that the conjecture is true under the additional assumption that $\nu$ is supported on a countable set. 

However, a recent work of the author \cite{Lim} presented a counterexample (see \cite[Example 2.9]{Lim}), showing that the assumption $\mu \ll \mathcal{L}^d$ alone is not sufficient. We think  that it is an interesting question to find some sufficient conditions on $\mu,\nu$ which guarantee the conjecture. Or, as commented in \cite{Lim}, for the following (weaker) existence conjecture I do not know a counterexample:\\

\noindent {\bf Conjecture 2:}  {\it  Consider the cost function 
$c(x, y) =|x-y|$  and assume that $\mu$ is absolutely continuous with respect to Lebesgue measure on $\R^d$, and that $\mu \wedge \nu =0$. Then there exists a martingale optimal transport to \eqref{MGTP} which is of polytope-type as described in the above conjecture.\\

The organization of the paper is as follows. In Section 2, we describe the monotonicity principle which was first introduced in \cite{bj} and subsequently generalized in \cite{Z, BeGr} and  \cite{bnt}. Then we establish the stability of the common marginal $\mu \wedge \nu$ under every minimizer of \eqref{MGTP}. In Section 3, we further apply the monotonicity to determine the structure of the minimizer in one dimension. Finally, in Section 4, we establish the deformation lemma and the main theorem which deals with arbitrary dimensions.

\section{Monotonicity principle and stability of $\mu \wedge \nu$ under every minimizer}
An important basic tool in optimal transport is the notion of \emph{$c$-cyclical monotonicity}. A parallel statement was given in \cite{bj}, then was generalized in \cite{Z}, \cite{BeGr}.

\begin{definition}\label{def:competitor}
Let $\sigma$ be a finite measure supported on a finite set $H \subset \R^d \times \R^d$. Let $X_H$ be the orthogonal projection of $H$ onto the first coordinate space $\R^d$. Then we say that $\rho$ is a \emph{competitor} of $\sigma$ if $\rho$ has  the same marginals as $\sigma$ and for each $x \in X_H$,
$\int_{\R^d} y\, d \sigma(x,y)=\int_{\R^d} y\, d \rho(x,y)$.
\end{definition}

\begin{lemma}[Monotonicity principle \cite{bj, Z, BeGr, bnt}]\label{mono}
Assume that $\mu, \nu$ are probability measures in convex order and that $c:\R^d \times \R^d \to \R$ is a Borel measurable cost function. Assume that $\pi\in {\rm MT}(\mu,\nu)$ is an optimal martingale transport plan which  leads to  finite cost.
Then there exists a Borel set $\Gamma \subset \R^d \times \R^d$ with $\pi(\Gamma)=1$ such that the following monotonicity principle holds: \\
If $\sigma$ is a finite measure on a finite set $H \subset \Gamma$, then for every competitor $\rho$ of $\sigma$, we have 
$$\int c\,d\sigma \leq \int c\,d\rho.$$
\end{lemma}
The meaning of the monotonicity principle is clear: $\supp (\sigma) \subseteq \Gamma$ means that $\sigma$ is a ``subplan" of the full transport plan $\pi$, and the definition of competitor means that if we change the subplan $\sigma$ to $\rho$, then the martingale structure of $\pi$ is not disrupted. Now if we have $\int c\,d\sigma > \int c\,d\rho$, then we may modify $\pi$ to have $\rho$ as its subplan, achieving less cost, therefore the current plan $\pi$ is not a minimizer. For more details and proofs, see  \cite[Lemma 1.11]{bj} or \cite[Theorem 3.6]{Z}.\\

The following notations are introduced in \cite{bj} and we use them in this paper:\\

\noindent For a set $\Gamma \subset \R^d \times \R^d$, 
we write $X_\Gamma :=$ proj$_X \Gamma$, $Y_\Gamma :=$ proj$_Y \Gamma$, i.e. $X_\Gamma$ is the projection of $\Gamma$ on the first coordinate space $\R^d$, and $Y_\Gamma$ on the second. 
For each $x\in \R^d$, we let
 $\Gamma_x =\{ y \in \R^d \ | \ (x, y) \in \Gamma\}$ and $\Gamma^x =\{ y \in \R^d \ | \ (y, x) \in \Gamma\}$ be the vertical and horizontal slices of $\Gamma$, respectively.
 
  The following definition will be also useful in this paper.
 
 \begin{definition}
 Let $\pi$ be a measure on $\R^d \times \R^d$ and let $\mu, \nu$ be  its marginals. We write $d \pi(x,y) = d \pi_x(y) d\mu(x)$ and $d \pi(x,y) = d \pi^y(x) d\nu(y)$ if $(\pi_x)_{x\in \R^d}$ and $(\pi^y)_{y\in \R^d}$ are  disintegrations of $\pi$ with respect to $\mu$ and $\nu$ respectively. 
  \end{definition}

Now as an application of the monotonicity principle, we prove the stability of $\mu \wedge \nu$ under every optimal martingale transport. \cite{bj} discusses the following theorem in one-dimensional setup with the Euclidean distance cost. We prove it here in general dimension with the class of cost functions \eqref{cost}. Note that radial symmetry of $\mu, \nu$ is not assumed.

\begin{theorem}\label{stability}
Let $\pi$ be any minimizer of the problem \eqref{MGTP} with cost \eqref{cost}. Then  the common mass $\mu \wedge \nu$ is stable under $\pi$, { in the sense that if we define $D: \R^d \to \R^d \times \R^d$ by $D(x) = (x,x)$, then the push-forward measure of $\mu \wedge \nu$ by the map $D$ is dominated by $\pi$, i.e. $D_\# (\mu \wedge \nu) \leq \pi$. } 
\end{theorem}
\begin{proof}   Suppose that the theorem is false, so that  there exists a minimizer $\pi$ such that $D_\# (\mu \wedge \nu) \nleq \pi$. Let $\pi_\Delta$ be the restriction of $\pi$ on the diagonal $\Delta = \{(x,x) \,|\, x \in \R^d\}$. Then since $D_\# (\mu \wedge \nu) \nleq \pi$, the measure $D_\# (\mu \wedge \nu) - \pi_\Delta$ has a non-zero positive part  and we let $\eta$ be the push-forwarded measure of this positive part by the map $(x,x) \mapsto x$. Let $d \pi(x,y) = d \pi_x(y) d\mu(x)$ and $d \pi(x,y) = d \pi^y(x) d\nu(y)$. Then by definition of $\eta$, we see that 
 \begin{align}\label{z1}
\pi_x \neq \delta_x\, \text{ and } \, \pi^x \neq \delta_x, \quad \eta-a.e. \, x.
\end{align}
Let $\Gamma$ be a monotone  set with $\pi(\Gamma)=1$ as in Lemma \ref{mono}.  By choosing a version of the disintegration $(\pi_x)_{x \in A}$ where $\mu(A)=1$, and by replacing $\Gamma$ with $\Gamma \cap (A \times \R^d)$, we can assume that $\delta_x \le_c \pi_x$ and $\pi_x(\Gamma_x)=1$ for every $x \in X_\Gamma$. This implies that, whenever $\pi_x \neq \delta_x$, we can find finitely many points  in $ \Gamma_x \setminus \{x\}$ such that $x$ can be written as a convex combination of them, that is 
 \begin{align}\label{z2}
x = \sum_{i=1}^n p_i y_i \quad \text{where} \quad y_i \in \Gamma_x \setminus \{x\}, \,\, p_i >0, \,\, \sum_{i=1}^n p_i =1.
\end{align}
Now \eqref{z1}, \eqref{z2} clearly imply that, for $\eta$ - a.e. $x$, we can find a probability measure $\rho_x$ such that $\delta_x \le_c \rho_x$, $\supp (\rho_x)$ is a finite subset of $\Gamma_x \setminus \{x\}$, and furthermore there is $z \in \R^d$, $z \neq x$ such that $(z,x) \in \Gamma$.

 Let us explain how this yields  a contradiction against the monotonicity of $\Gamma$. In \eqref{cost} since $f$ is increasing and concave we have $f(|x-y|) +f( |z-x|) \geq f(|y-z|)$, and in fact the inequality is strict unless $z,x,y$  lie on a line in this order. Hence, whenever $z\neq x$, $\rho_x \neq \delta_{x}$ and $\delta_x \le_c \rho_x$, we have
\begin{align*}
\int f(|x-y|)\, d\rho_x (y) + f(|z-x| ) > \int f(|y-z|)\, d\rho_x (y).
\end{align*}
This may be rephrased as follows: the cost of  ``sending the mass $\delta_z$ to $ \rho_x$ and $\delta_x$ to $\delta_x$" is cheaper than ``sending $\delta_z$ to $\delta_x$ and $\delta_x$ to $\rho_x$" (recall $f(0)=0$). As $(z,x) \in \Gamma$ and $\supp (\rho_x)$ is a finite subset of $\Gamma$, this contradicts to the fact that $\Gamma$ is monotone.  
\end{proof}

\begin{remark}\label{stayput1}
It is clear from the proof that Theorem \ref{stability} holds with cost $c$ if
\begin{align*}
\int c(x,y) \,d\rho_x (y) + c(z,x) > \int c(y,z) \,d\rho_x (y)
\end{align*}
holds whenever $z\neq x$, $\rho_x \neq \delta_{x}$ and $\delta_x \le_c \rho_x$. In particular, Theorem \ref{stability} holds when $c(x,y) = ||x-y||$ where $|| \cdot ||$ is a strictly convex norm on $\R^d$.  (A norm is strictly convex if the unit ball is strictly convex, that is every point of the boundary of the unit ball is an extreme point. For example, the Euclidean norm is strictly convex as its unit ball is ``completely round".) 

On the other hand, Theorem \ref{stability} fails for general norm costs; for example, let $d=2$, $||(x,y)|| = |x|+|y|$ and let $\mu=\frac{1}{2} ( \delta_{(0,0)} +  \delta_{(0,1)})$, $\nu=\frac{1}{4}( \delta_{(0,0)} + \delta_{(-1,0)} + \delta_{(1,0)} +  \delta_{(0,2)})$. It is easy to check that then every martingale transport between $\mu, \nu$  yields the same cost. 
\end{remark}

By Theorem \ref{stability}, we can reduce the marginals of the problem \eqref{MGTP} to the  disjoint marginals $\bar{\mu}:=\mu - \mu \wedge \nu$ and $\bar{\nu}:=\nu - \mu \wedge \nu$. Thus, from now on we will always assume that $\mu \wedge \nu = 0$, and therefore, for any minimizer $\pi \in {\rm MT}(\mu, \nu)$, we have a monotone set $\Gamma$ such that $\pi(\Gamma) =1$ and $\Gamma \cap \Delta = \emptyset$, where $\Delta := \{(x,x) \,|\, x \in \R^d\}$.


\section{Structure of optimal martingale transport in one dimension}

In this section, we study the problem \eqref{MGTP} in one dimension, i.e. the marginals $\mu, \nu$ are defined on the real line $\R$. We will consider the cost function \eqref{cost} and will determine the structure of optimal coupling. In this section, we do not assume the symmetry of marginals $\mu, \nu$ with respect to the origin. Recall that we can assume $\mu \wedge \nu =0$. Finally, we will say that $\mu$ is continuous if $\mu$ does not assign positive measure at any point: $\mu(\{x\}) = 0$ for every $x \in \R$.\\

The following theorem for the 1-dimensional case  was shown in \cite{HoKl12,bj} when the cost is the Euclidean distance $c(x,y)=|x-y|$.    By closely following the idea presented in \cite{bj}, here we extend it for the class of costs \eqref{cost} for the application to the main theorem in the next section.
\begin{proposition}\label{1d}
Assume that $\mu \wedge \nu =0$ and $\mu$ is continuous. Let $\pi$ be a minimizer for the problem $\eqref{MGTP}$ with $d=1$ and the cost \eqref{cost}. Then, there exists a monotone set $\Gamma$ such that $\pi(\Gamma)=1$ and for every $x \in X_\Gamma$, we have $|\Gamma_x| = 2$. Hence if we define two functions $S : X_\Gamma \to \R$ and $T : X_\Gamma \to \R$ by $\Gamma_x = \{S(x), T(x)\}$ and $S(x) < x <T(x)$, then $\pi$ is concentrated on $graph(S) \cup graph(T)$. In particular, the minimizer is unique.
\end{proposition}
\begin{proof}
Let $\Gamma$ be any monotone set of $\pi$ with $\pi(\Gamma)=1$ and suppose\\
 $(x, y^-), (x, y^+), (x', y') \in \Gamma$, with $y^-< y' < y^+$. Then we claim that neither
$y^-<x'<x\leq y'$ nor $y'\leq x<x'<y^+$ is possible. To prove the claim, suppose $y^-<x'<x\leq y'$ and let $0 < t <1$ be such that $ty^- + (1-t)y^+ = y'$. Now consider the function
\begin{align*}
G(z) = t c(z,y^-) + (1-t) c(z,y^+) - c(z,  y').
\end{align*}
If $y^- < z <y'$, this becomes (recall $c(x,y) = f(|x-y|)$)
\begin{align*}
G(z) = t f(z - y^- ) + (1-t) f(y^+ - z) - f(y' - z).
\end{align*}
By taking derivative, we get
\begin{align*}
G'(z) &= t f'(z - y^- ) - (1-t) f'(y^+ - z) + f'(y' - z)\\
&= t [f'(z - y^- ) + f'(y^+ - z)] + [f'(y' - z) - f'(y^+ - z)].
\end{align*}
We observe that $f'(y' - z) - f'(y^+ - z) \ge 0$ and $f'(z - y^- ) + f'(y^+ - z) >0$, thus $G'(z) >0$.
Hence for $y^-<x'<x\leq y'$ we have $G(x') < G(x)$, that is
\begin{align*}
t c(x', y^- )+ (1-t) c(x', y^+) + c(x,y') <  t c(x, y^- )+ (1-t) c(x, y^+) + c(x',y').
\end{align*}
This means that if we define a measure $\sigma$ by $\sigma = t \delta_{(x, y^-)} + (1-t) \delta_{(x, y^+)} + \delta_{(x', y')}$, then we have a cost-efficient competitor $\rho$ by defining $\rho = t \delta_{(x', y^-)} + (1-t) \delta_{(x', y^+)} + \delta_{(x, y')}$. Note that $\rho$ satisfies the assumption to be a competitor of $\sigma$. Hence by Lemma \ref{mono}, $(x, y^-), (x, y^+), (x', y') \in \Gamma$ with $y^-< y' < y^+$ and $y^-<x'<x\leq y'$ cannot occur. The case $y'\leq x<x'<y^+$ cannot occur by similar reasoning.

Now we follow the argument in \cite{bj}: Suppose  the set $A:=\{x \in \R : |\Gamma_x| \geq 3 \}$ is uncountable. ($|\Gamma_x| $ is the cardinality of the set $\Gamma_x$.) Then we will have $(x, y^-), (x, y^+), (x, y) \in \Gamma$, with $y^- < x < y < y^+$ or $y^- < y < x < y^+$ (Recall that $\Gamma \cap \Delta = \emptyset$, where $\Delta := \{(x,x) \,|\, x \in \R^d\}$, since $\mu \wedge \nu =0$). Assume the first case. Then the Lemma 3.2 in \cite{bj} shows that for any given $\epsilon >0$, we have $(x', y') \in \Gamma$ with $x-\epsilon < x' <x$ and $|y' - y| < \epsilon$ by the uncountability of $A$. Then for small $\epsilon$ we have the first forbidden case, and similarly if $y^- < y < x < y^+$ then we have $(x', y') \in \Gamma$ with $x< x' <x+\epsilon $ and $|y' - y| < \epsilon$, the second forbidden case, a contradiction. Hence $A$ must be countable, therefore by continuity of $\mu$, $A$ is negligible. 

Uniqueness follows by the usual argument, namely, if $\pi_1$ and $\pi_2$ are optimal solutions realized by $(S_1, T_1)$ and $(S_2, T_2)$ respectively, then the average $\frac{\pi_1 + \pi_2}{2}$ is also optimal and hence it must also be realized by two functions $(S,T)$. This implies that $S_1(x) = S_2(x)$ and $T_1(x) = T_2(x)$ for $\mu$ a.e. $x$, yielding uniqueness.
\end{proof}  
In fact, we can say more on the structure of optimal martingale couplings. We note that the rest of this section is largely motivated by \cite{HoKl12}. While \cite{HoKl12} gives a very detailed study, e.g. the construction of a martingale coupling and duality result for the cost $c(x,y)=|x-y|$, we shall be brief here and focus on the uniqueness property of the solution to \eqref{MGTP} for the class of costs \eqref{cost} for the main theorem in the next section. Note that in the rest of this section, we do not assume the continuity of $\mu$.
\begin{lemma}\label{decrease}
Let $I$ be a bounded  interval with boundary $\partial I = \{a,b\}$ and suppose $\nu(I)=0$. Let $\pi$ be a minimizer for the problem $\eqref{MGTP}$ with respect to the cost \eqref{cost}.  Let $d \pi(x,y) = d \pi_x(y) d\mu(x)$ and if $x \in I$, then denote $\pi^+_x$ as the restriction of $\pi_x$ on $[b, \infty)$ and $\pi^-_x$ as the restriction of $\pi_x$ on $(-\infty, a]$.  Then $x, x' \in I$ and $x < x'$ implies
$$\sup(\supp(\pi^+_{x'})) \leq \inf(\supp(\pi^+_x)) \quad \text{and} \quad \sup(\supp(\pi^-_{x'})) \leq \inf(\supp(\pi^-_x)).$$
 In other words, the set-valued functions $x \mapsto \supp(\pi^+_x)$ and $x \mapsto \supp(\pi^-_x)$ decrease on $I$.
\end{lemma}

\begin{proof}
Let $\Gamma$ be a monotone set of $\pi$ with $\pi(\Gamma)=1$ and $Y_\Gamma \cap I = \emptyset$. If $x, x' \in I \cap X_\Gamma$ and $x < x'$, then we claim that $\sup(\Gamma^+_{x'}) \leq \inf(\Gamma^+_x)$, where $\Gamma^+_x := \Gamma_x \cap [b, \infty)$. If not, then we can find $y' > y \geq b$ such that $(x,y), (x', y') \in \Gamma$. As $\pi$ is a martingale, we can also find $y'' \leq a$ with $(x', y'') \in \Gamma$. Then the configuration $(x,y), (x', y'), (x', y'') \in \Gamma$ is forbidden by the proof of Theorem \ref{1d}, a contradiction. As $\pi_x(\Gamma_x) = 1$ for every $x \in X_\Gamma$, $\pi^+_x$ has its full mass on $\Gamma^+_x$, hence $\sup(\supp(\pi^+_{x'})) \leq \inf(\supp(\pi^+_x))$. The other case $\sup(\supp(\pi^-_{x'})) \leq \inf(\supp(\pi^-_x))$ can be proved similarly.
\end{proof}

We may call the above result as ``local decreasing property", as the function $x \mapsto \supp(\pi^+_x)$ and $x \mapsto \supp(\pi^-_x)$ decrease locally, i.e. on any  interval $I$ where $\nu(I)=0$. Thus if we make the following assumption, we will have the global decreasing property for any optimal martingale transport. This assumption was also made in \cite{HoKl12}.\\

\noindent {\bf Dispersion Assumption.} \,\,There is an  interval $I$ such that
 $$\mu(I) = 1 \quad \text{and} \quad \nu(I) = 0.$$

For example, two Gaussian measures $\mu, \nu$ in convex order will satisfy this assumption, after $\mu \wedge \nu$ is subtracted from each marginal. Now we observe that the global decreasing property also yields the uniqueness of optimal solution, without assuming the continuity of $\mu$.

\begin{proposition}\label{uniqueness}
 Assume the dispersion assumption. Then there exists a unique element in ${\rm MT}(\mu,\nu)$ which is  decreasing in the sense of Lemma \ref{decrease}. Similarly, there exists a unique increasing element  in ${\rm MT}(\mu,\nu)$. In particular, under the dispersion assumption there is a unique solution to the problem \eqref{MGTP} with the cost \eqref{cost}.
\end{proposition} 
\begin{proof} We shall only need to prove that the decreasing property uniquely determines the martingale $\pi \in {\rm MT}(\mu,\nu)$, as the increasing case will be similar.

Let $I$ be a bounded  interval with boundary $\partial I = \{a,b\}$ and let $p^2: \R^2 \to \R$ be the projection to the second coordinate. For $x\in \R$ and $\pi \in {\rm MT}(\mu,\nu)$, define 
$$\nu^-_{\pi,x} =p^2_\# (\pi\big{|}_{(-\infty, x] \times (-\infty, a]}), \quad\nu^+_{\pi,x} =p^2_\# (\pi\big{|}_{(-\infty, x] \times [b, \infty)}).$$
Let $ \nu_{\pi,x}= \nu^-_{\pi,x}+\nu^+_{\pi,x}$ and note that for  $\pi, \tilde \pi \in {\rm MT}(\mu,\nu)$,  $\pi = \tilde \pi \Leftrightarrow \nu_{\pi,x} = \nu_{\tilde \pi,x}$ for every $x \in \R$. Now assume that $\pi, \tilde \pi$ are optimal solutions for the problem \eqref{MGTP} and fix $x \in \R$. We claim that $\nu_{\pi,x} = \nu_{\tilde \pi,x}$.

To see this, observe that Lemma \ref{decrease} clearly implies that $\nu^-_{\pi,x}, \nu^+_{\pi,x}$ must be    concentrated to the right; that is, $\nu^-_{\pi,x}, \nu^+_{\pi,x}$ must be of the form
\begin{align}\label{right}
\nu^-_{\pi,x}=\nu\big{|}_{(s,a]}+c^-\nu(\{s\})\delta_s, \quad \nu^+_{\pi,x}=\nu\big{|}_{(t,\infty)}+c^+\nu(\{t\})\delta_t
\end{align}
for some $s,t \in \R$ and $0 < c^+, c^- \le 1$. Then since $\mu((-\infty, x]) = ||\nu^-_{\pi,x}||+||\nu^+_{\pi,x} ||= ||\nu^-_{\tilde \pi,x}||+||\nu^+_{\tilde \pi,x}||$, the right-concentration property \eqref{right} implies that either 
$$\nu^-_{\pi,x} \ge \nu^-_{\tilde \pi,x}\, \text{ and }  \, \nu^+_{\pi,x} \le \nu^+_{\tilde\pi,x}, \quad \text{or} \quad \nu^-_{\pi,x} \le \nu^-_{\tilde \pi,x}\, \text{ and }  \, \nu^+_{\pi,x} \ge \nu^+_{\tilde\pi,x}.$$
Let us assume the first case. Now the fact that $\pi$ being a martingale measure implies in particular
\begin{align*} 
 \int y \,d\mu\big{|}_{(-\infty,x]}(y) = \int y \,d(\nu_{\pi,x}) 
\end{align*}
and of course $ \int y \,d\mu\big{|}_{(-\infty,x]}(y) = \int y \,d(\nu_{\tilde\pi,x})$ also. Hence
\begin{align*}
& \int y \,d(\nu_{\pi,x} - \nu_{\tilde \pi,x})(y) = \int y \,d\mu\big{|}_{(-\infty,x]}(y) -  \int y\, d\mu\big{|}_{(-\infty,x]}(y) = 0\\
& \Rightarrow \int y \, d(\nu^-_{\pi,x} -\nu^-_{\tilde \pi,x})(y) = \int y \, d(\nu^+_{\tilde \pi,x} -\nu^+_{ \pi,x})(y).
\end{align*}
But as we assumed the first case, both $\nu^-_{\pi,x} -\nu^-_{\tilde \pi,x}$, $\nu^+_{\tilde \pi,x} -\nu^+_{ \pi,x}$ are nonnegative measures which are concentrated on the disjoint intervals $(-\infty,a]$, $[b, \infty)$ respectively. Therefore, the last identity implies that both measures must be zero, i.e. $\nu^-_{\pi,x} = \nu^-_{\tilde \pi,x}$ and  $\nu^+_{\tilde \pi,x} = \nu^+_{ \pi,x}$. The claim is proved, hence the proposition.
\end{proof}

In particular if we assume the symmetry of $\mu, \nu$ with respect to the origin:
\begin{corollary}\label{coro}
 If $\mu, \nu$ are symmetric with respect to the origin, then under the assumptions of Theorem \ref{1d} or \ref{uniqueness}, the optimal martingale coupling  is unique and  symmetric with respect to the origin.
\end{corollary}

\begin{proof}
We can prove it directly, or we let $\eta$ be the optimal coupling in Theorem \ref{1d} or \ref{uniqueness}, and let 
$\zeta = \frac{1}{2}(\eta + \eta')$ be a symmetrization of $\eta$, where $\eta'$ is the reflection of $\eta$ with respect to the origin. Then $\zeta$ is also optimal, so by uniqueness, $\eta = \zeta$.
\end{proof}

\section{Structure of optimal martingale transport in higher dimensions}
We have studied the structure of the martingale transport in one dimension which minimizes $\E \,f(|X-Y|)$, and in particular have shown its uniqueness either when $\mu$ is continuous or when the separation assumption holds. In this section, we will introduce the notion of symmetrization of a transport plan, and then will present a variational calculus which will reduce the higher dimensional problem under radially symmetric marginals to the one-dimensional situation.

\subsection{Symmetrization of transport plans and the $R$-equivalence relation} 

In this section, we introduce the notion of $R$-equivalence on the space of probability measures on $\R^d$ and the notion of symmetrization of transport plans (i.e. probability measures on $\R^d \times \R^d$). These ideas will play a crucial role  for the proof of Theorem \ref{mainresult}. First, we introduce the notion of $R$-equivalence on the space of probability measures on $\R^d$. Let $m(x)=|x|$ be the modulus map on $\R^d$.

\begin{definition} Probability measures $\sigma$ and $\rho$ on $\R^d$ are called $R$-equivalent if $m_\#(\sigma) = m_\#(\rho)$, i.e. $\sigma$ and $\rho$ contain the same mass on any annulus. That is, for any $B \subset \R_+$ and any $A_B:= \{ x \in \R^d \,|\, |x| \in B \}$ we have $\sigma (A_B) = \rho (A_B)$. In this case we write $\sigma \cong_R \rho$. We define $\mathcal{R}\sigma$ as the unique radially symmetric probability measure satisfying $\mathcal{R}\sigma \cong_R \sigma$. 
\end{definition}
 As there is one-to-one correspondence between radially symmetric measures on $\R^d$ and measures on $\R_+$ by the push-forward map $m_\#$ (i.e. any radially symmetric measure $\sigma$ is characterized by $m_\#(\sigma)$), the definition of $\mathcal{R}\sigma$ is justified.


Let $O(d)$ be the orthogonal group in dimension $d$,  on which the Haar measure  $\mathcal{H}$ is defined. Given $M \in O(d)$ and a transport plan $\pi$,  we define the push-forward $M_\#\pi$ as the following: for Borel sets $C \subset \R^{2d}$, we define
$$M_\# \pi(C) = \pi(M^{-1}(C))$$
where $(x,y) \in M^{-1}(C) \Leftrightarrow (Mx,My) \in C$ for all $x,y \in \R^d$. In particular, $M_\# \pi(A \times B) = \pi(M^{-1}(A) \times M^{-1}(B))$ for all $A,B \subset \R^d$.

Note that if $\pi \in \Pi(\mu, \nu)$, then $M_{\#}\pi \in \Pi(M_{\#}\mu, M_{\#}\nu)$. Now we introduce the symmetrization operator which acts  on the space of transport plans.
\begin{definition}\label{symmetrization}
We define the symmetrization operator $\mathcal{S}$ on a transport plan $\pi$ as: for each Borel subset $D \subset \R^{2d}$,
$$\mathcal{S}\pi(D) =\int_{M \in O(d)}  M_{\#}\pi(D) \,d\mathcal{H}(M).$$
\end{definition}
 \begin{definition}
Let $0 \neq x\in \R^d$ and let $L_x$ be the one-dimensional subspace spanned by $x$. Let $O_x(d):= \{ M \in O(d)\,|\, Mx=x\}$. We say that a probability measure $\sigma$ on $\R^d$ is $L_x$-symmetric if for every $A \subset \R^d$ and $M \in O_x(d)$ we have 
$$\sigma(A) = \sigma(M(A)).$$
\end{definition}

The following lemma explains why we want to consider the operator $\mathcal{S}$. 

\begin{lemma}\label{symmetryproperty} $\mathcal{S}$ has the following properties:
\begin{enumerate}
\item If $\pi \in \Pi(\mu, \nu)$, then $\mathcal{S}\pi \in \Pi(\mathcal{R}\mu, \mathcal{R}\nu)$. In particular if $\pi$ has radially symmetric marginals, then $\mathcal{S}\pi$ has the same marginals as $\pi$.

\item Let $\pi \in \Pi(\mu, \nu)$ and let $d \mathcal{S}\pi (x,y) = d \mathcal{S}\pi _x(y) d \mathcal{R}\mu(x)$. Then $\mathcal{S}\pi _x$ is $L_x$-symmeric. Furthermore $(\mathcal{S}\pi _x)_x$ is rotationally congruent;  for every $x,y \in \R^d$ and $M \in O(d)$ such that $Mx=y$, we have $M_\#\mathcal{S}\pi _x = \mathcal{S}\pi _y$.

\item If $\pi \in {\rm MT}(\mu, \nu)$, then $\mathcal{S}\pi \in {\rm MT}(\mathcal{R}\mu, \mathcal{R}\nu)$.

\item If the cost function is rotation invariant, that is if $c(x,y) = c(Mx, My)$ for any $M \in O(d)$, then the cost of a plan $\pi$ is equal to the cost of $\mathcal{S}\pi$.
\end{enumerate}
\end{lemma}
\begin{proof}
\begin{enumerate}
\item Let $A\subset\R^d$ and $N \in O(d)$. Then 
\begin{align*}
\mathcal{S}\pi (N(A) \times \R^d) &=\int_{M \in O(d)}  M_{\#}\pi(N(A) \times \R^d) \,d\mathcal{H}(M)\\
&=\int_{M \in O(d)}  \pi(M^{-1}N(A) \times \R^d) \,d\mathcal{H}(M)\\
&=\int_{M \in O(d)}  \pi(M^{-1}(A) \times \R^d) \,d\mathcal{H}(M)\\
&=\int_{M \in O(d)}  M_{\#}\pi(A \times \R^d) \,d\mathcal{H}(M)\\
&=\mathcal{S}\pi (A \times \R^d)
\end{align*}
showing that the marginals of $\mathcal{S}\pi$ are radially symmetric. Now let $B \subset \R^d$ be an annulus set, that is $B = M(B)$ for any $M \in O(d)$. Then
\begin{align*}
\mu(B)&=\pi(B \times \R^d) = \pi(M^{-1}(B) \times \R^d)\\
&=\int_{M \in O(d)}  \pi(M^{-1}(B) \times \R^d) \,d\mathcal{H}(M)\\
&=\int_{M \in O(d)}  M_{\#}\pi(B \times \R^d) \,d\mathcal{H}(M)\\
&=\mathcal{S}\pi (B \times \R^d)
\end{align*}
showing that $\mathcal{S}\pi \in \Pi(\mathcal{R}\mu, \mathcal{R}\nu)$.

\item Let $A \subset \R^d$, $B_{x,r}$ be the open ball centered at $x$ and radius $r$, and $N \in O_x(d)$. Then 
\begin{align*}
\mathcal{S}\pi (B_{x,r} \times N^{-1}(A)) &=\int_{M \in O(d)}  M_{\#}\pi(B_{x,r} \times N^{-1}(A)) \,d\mathcal{H}(M)\\
&=\int_{M \in O(d)}  \pi(M^{-1}N^{-1}(B_{x,r}) \times M^{-1}N^{-1}(A)) \,d\mathcal{H}(M)\\
&=\int_{M \in O(d)} M_\# \pi(B_{x,r} \times  A) \,d\mathcal{H}(M)\\
&=\mathcal{S}\pi (B_{x,r} \times  A).
\end{align*}
Dividing by $\mu(B_{x,r})$ and letting $r \to 0$, as $\mathcal{S}\pi_x(A)$ being the Radon-Nikodym derivative $\frac{d\mathcal{S}\pi (x,A)}{d \mu(x)}$, we get that  $\mathcal{S}\pi _x(A) = \mathcal{S}\pi _x(N^{-1}(A))$, i.e. $\mathcal{S}\pi _x$ is $L_x$-symmetric. Now let $x,y \in \R^d$ and $M \in O(d)$ such that $Mx=y$. Then the rotational congruence of disintegration is clear once we observe that $\mathcal{S}\pi (B_{x,r} \times A) = \mathcal{S}\pi (B_{y,r} \times M(A))$ for any $A \subset \R^d$.

\item Let $h:\R^d \to \R^d$ be a bounded measurable function. Then 
\begin{align*}
&\int_{\R^{2d}} h(x)\cdot(y-x) \,d\mathcal{S}\pi (x, y)\\
=&\int_{O(d)}\int_{\R^{2d}} h(x)\cdot(y-x) \,dM_\#\pi (x, y)\,d\mathcal{H}(M) = 0
\end{align*}
since $M_\#\pi$ is clearly a martingale measure. This proves that $\mathcal{S}\pi$ is a martingale measure.

\item Since the cost of $\pi$ is the same as the cost of $M_\# \pi$ for any $M \in O(d)$, it is clear from the definition of $\mathcal{S}\pi$.
\end{enumerate} 
\end{proof}
The meaning of the operator $\mathcal S\pi$ is that disintegration of $\mathcal S\pi$ is made by superposition of that of $\pi$ moved by all orthogonal matrices. Let us give an example.
\begin{example}
Let $d=2$, $\mu = \frac{3}{5} \delta_{(0,1)} + \frac{2}{5} \delta_{(1,0)}$, and $\pi$ be a transport plan with the first marginal $\mu$ and the disintegration $\pi_{(0,1)} = \delta_{(0,2)}$, $\pi_{(1,0)} =  \delta_{(1,1)}$ with respect to $\mu$. In words, $\pi$ transports 3/5 mass from $(0,1)$ to $(0,2)$ and 2/5 mass from $(1,0)$ to $(1,1)$. What is the first marginal and disintegration of $\mathcal{S}\pi$? As $\mu$ has its total mass on the unit circle, the first marginal of $\mathcal{S}\pi$ is the uniform probability measure on the unit circle, that is $\mathcal{R}\mu$. Next, let us  superpose (by rotation) the two transports occured by $\pi$ to have the same (but any) starting point on the unit circle, say we choose the starting point as $(0,1)$, so the resulting superposed plan transports 3/5 mass from $(0,1)$ to $(0,2)$ as before and 2/5 mass from $(0,1)$ to $(-1,1)$. Let us write this as $\tilde \pi_{(0,1)} = \frac{3}{5} \delta_{(0,2)} + \frac{2}{5} \delta_{(-1,1)}$. Next, we symmetrize $\tilde \pi_{(0,1)}$ with respect to its ``axis" $(0,1)$, so that the resulting transport becomes $(\mathcal{S}\pi)_{(0,1)} = \frac{3}{5} \delta_{(0,2)} + \frac{1}{5} \delta_{(-1,1)} + \frac{1}{5} \delta_{(1,1)}$. This is a disintegration of $\mathcal{S}\pi$ at the point $(0,1)$, and disintegration at other points of the unit circle is simply the rotation of $(\mathcal{S}\pi)_{(0,1)}$. For example, $(\mathcal{S}\pi)_{(1,0)} = \frac{3}{5} \delta_{(2,0)} + \frac{1}{5} \delta_{(1,1)} + \frac{1}{5} \delta_{(1,-1)}$.
\end{example}

\subsection{Deformation lemma and main theorem} In this section, we will present a deformation lemma which will allow martingale transport problem under radial marginals to be reduced to the problem on the one-dimensional subspaces, where we can apply the results in the previous section. 

For a function $f$ and a measure $\mu$, we denote $f\mu$ to be the measure defined by $f\mu(A) = \int_A f(x) d\mu(x)$. We use this notation with  functions $w^\pm$ and a measure $\sigma_*$ in \eqref{deform} of the next lemma.
 \begin{lemma}\label{variational}
Consider the cost function of the form $c(x,y) = h(|x-y|)$ and let $L_x$ be the one-dimensional subspace spanned by $x$, $x\neq 0$. Let $\sigma$ be a $L_x$-symmetric probability measure on $\R^d$ with barycenter at $x$. Define continuous functions $f_x^+, f_x^- : \R^d \to \R^d$ by $ f_x^- = -f_x^+$, and
\begin{align}\label{f+f-}
f_x^+(y) = \frac{|y|}{|x|}x \quad \text{if} \quad y \notin L_x, \quad f_x^+(y) = y \quad \text{otherwise.} \end{align}
Suppose that $r \mapsto {h'(r)}/{r}$ is strictly decreasing for $r>0$. If $\sigma(L_x) <1$, then there exists a probability measure  $\rho$ with barycenter at $x$, $\rho(L_x)=1$ and $\rho \cong_R  \sigma$, such that
\begin{align}\label{costless}
\int_{\R^d} h(|x-y|)\, d\sigma(y) > \int_{\R^d} h(|x-y|)\, d\rho(y).
\end{align}
 A concrete choice of such a $\rho$ is as follows: write $\sigma = \sigma_* + \sigma_0$ where $\sigma_0 = \sigma(0)\delta_0$. Let $w^+ (y)= \frac{1}{2}+\frac{1}{2} \frac{x}{|x|}\cdot \frac{y}{|y|}$, $w^- (y)= \frac{1}{2}-\frac{1}{2} \frac{x}{|x|}\cdot \frac{y}{|y|}$ for $y \neq 0$. Then
\begin{align}\label{deform}
\rho = L_x(\sigma) := \sigma_0 + {f_x^+}_\#(w^+\sigma_*) + {f_x^-}_\#(w^-\sigma_*).
\end{align}
\end{lemma}

\begin{remark}
As $(\frac{h'(r)}{r})' = \frac{rh''(r) - h'(r)}{r^2}$, the cost \eqref{cost} satisfies the assumption of the lemma. Another example is  $h(r)=r^p$, $0<p<2$, or $h(r)= -r^p$, $p>2$.
\end{remark}

\begin{proof} We will first prove the lemma in the simple case where $d=2$ and $\sigma$ is supported at two points. Now we will explain how to deform $\sigma$ to obtain $\rho$.

For this purpose, we will consider the family of probability measures $\rho(t)$ supported on the four points $z^1_n (t), z^2_n (t), z^1_s (t), z^2_s (t)$ in $\R^2$ (which we will define below), where $0 \leq t \leq 1$ is a parameter. We will identify the measure $\rho(t)$ with the transport plan from $\delta_x$ to $\rho(t)$, and we will observe that the cost of the transport $\rho(t)$ strictly decreases as $t$ increases. This will be the desired deformation process and $\rho(1)=\rho$  in the lemma while $\rho(0)= \sigma$.

To begin, without loss of generality let the barycenter $x$ be a point in $\R^2$, $x=(0,b), b\neq 0$. Let $z^1, z^2 \in \R^2$, $|z^1|=|z^2|=r>0$ and let 
$z^1 = (a, z), \,z^2 = (-a, z)$. Now for $0\leq t \leq 1$, let $z_n(t) = z + t (r-z), \,z_s(t) = z - t (r+z)$, and let
\begin{align*}
&z^1_n (t) = \big{(}\sqrt{r^2 - (z_n (t))^2}, z_n (t) \big{)}, \quad z^2_n (t) = \big{(}-\sqrt{r^2 - (z_n (t))^2}, z_n (t) \big{)}, \\ &z^1_s (t) = \big{(}\sqrt{r^2 - (z_s (t))^2}, z_s (t) \big{)}, \quad z^2_s (t) = \big{(}-\sqrt{r^2 - (z_s (t))^2}, z_s (t) \big{)}.
\end{align*}

Thus, the four points $z^1_n (t), z^2_n (t), z^1_s (t), z^2_s (t)$ are on the circle of center $0$ and radius $r$, and they are symmetrically located with respect to $L_x$. Now define the probability measure $\rho(t)$ and the transport cost from $\delta_x$ to $\rho(t)$
\begin{align*}
\rho(t) &= \frac{r+z}{4r} \delta_{z^1_n (t)} + \frac{r+z}{4r} \delta_{z^2_n (t)} + \frac{r-z}{4r} \delta_{z^1_s (t)} + 
\frac{r-z}{4r} \delta_{z^2_s (t)},\\
C(t) &= \frac{r+z}{2r} h (|z^1_n (t) - x|) + \frac{r-z}{2r} h (|z^1_s (t) - x|).
\end{align*}


Note that $\rho(0) = \frac{1}{2} \delta_{(-a, z)} + \frac{1}{2} \delta_{(a, z)} = \sigma$ and $\rho(1) = \frac{r+z}{2r} \delta_{(0, r)} + \frac{r-z}{2r} \delta_{(0, -r)} = \rho$, so $\rho (t)$ is a continuous deformation from 
$\sigma$ to $\rho$ along the circle of radius $r$. The point is that for all $0 \leq t \leq 1$, the barycenter of $\rho (t)$ is fixed at $(0, z)$ and they are obviously $R$-equivalent. Now we will show $C'(t) <0$ if ${h'(r)}/{r}$ is strictly decreasing for $r>0$. To see this, we compute
\begin{align*}
C'(t) &= \bigg{(}\frac{r+z}{2r}\bigg{)} \frac{h'(|z^1_n (t) - x |)}{|z^1_n (t) - x |} \big{<}z^1_n (t) - x , \frac{d}{dt} \big{(}z^1_n (t) \big{)}\big{>} \\
&+ \bigg{(}\frac{r-z}{2r}\bigg{)} \frac{h'(|z^1_s (t) - x |)}{|z^1_s (t) - x |} \big{<}z^1_s (t) - x , \frac{d}{dt} \big{(}z^1_s (t) \big{)}\big{>}
\end{align*}
where $\big{<},\big{>}$ is the inner product. Note  $\big{<}z^1_n (t), \frac{d}{dt} \big{(}z^1_n (t) \big{)}\big{>} = 
\big{<}z^1_s (t), \frac{d}{dt} \big{(}z^1_s (t) \big{)}\big{>} = 0$, and
\begin{equation*}
\big{<}x, \frac{d}{dt} \big{(}z^1_n (t) \big{)}\big{>} = b (r - z),\quad  \big{<}x, \frac{d}{dt} \big{(}z^1_s (t) \big{)}\big{>} = -b(r + z), \quad \text {hence}
\end{equation*}
\begin{equation*}
C'(t) = \bigg{(}\frac{b(r+z)(r-z)}{2r}\bigg{)} \bigg{[} \frac{h'(|z^1_s (t) - x |)}{|z^1_s (t) - x |} - 
\frac{h'(|z^1_n (t) - x |)}{|z^1_n (t) - x |} \bigg{]}.
\end{equation*}
Now we compute
\begin{align*}
&|z^1_n (t) - x |^2 = r^2 + b^2 -2b\,z_n(t) = r^2 + b^2 -2b (z+t(r-z))\\
&|z^1_s (t) - x |^2 = r^2 + b^2 -2b\,z_s(t) = r^2 + b^2 -2b (z-t(r+z))\\
&|z^1_n (t) - x |^2 - |z^1_s (t) - x |^2 = -4brt.
\end{align*}
Hence, we see that 
\begin{align*}
&|z^1_n (t) - x | < |z^1_s (t) - x | \quad \text{if} \quad b>0\\
&|z^1_n (t) - x | > |z^1_s (t) - x | \quad \text{if} \quad b<0.
\end{align*}
Thus in any case $C'(t) < 0$. Hence, $C(0) > C(1)$ and observe that $\rho$ is supported on the line $L_x$, and in fact $\rho(1) = L_x(\rho(0))$ as defined in \eqref{deform}.

Now we turn to the general case, but by the $L_x$-symmetry of $\sigma$ we will observe that the above case is already sufficiently general. The point is that for any $z \in \R^d$ there exists a unique reflected point $\tilde z$ with respect to $L_x$, and $\sigma$ assigns the same mass around $z$ and $\tilde z$, i.e. for $A\subset \R^d$ and its reflection $\tilde A$ ($z \in A \Leftrightarrow \tilde z \in \tilde A$) we have $\sigma(A)= \sigma (\tilde A)$. Hence the $L_x$-symmetric measure $\sigma$ can be seen as the sum of infinitesimal measures supported on such symmetric pairs $(z, \tilde z)$, to each of which the above variational calculus applies. 

A detailed calculation is as follows:
\begin{align*}      
&\int h(|x-z|)\, d\sigma(z) \\
=&\int \frac{1}{2}\bigg{(}h(|x-z|)+   h(|x-\tilde z|) \bigg{)}\, d\sigma(z)  \\
=&\iint h(|x-y|)\, d\rho_z(y)\,d\sigma(z) \quad\text{(where $\rho_z := \frac{1}{2}(\delta_z + \delta_{\tilde z})$)}.
\end{align*}
Now whenever $z \notin L_x$, the computation we presented above (applied on the 2-dimensional subspace generated by $x$ and $z$) shows that
\begin{align*}      
\int h(|x-y|)\, d\rho_z(y) > \int h(|x-y|)\, dL_x(\rho_z)(y).
\end{align*}
Thus if $\sigma(L_x) <1$, 
\begin{align*}      
&\int h(|x-y|)\, d\sigma(y) 
>  \iint h(|x-y|)\, dL_x(\rho_z)(y)\,d\sigma(z)\\
= &\iint h(|x-y|)\, dL_x(\delta_z)(y)\,d\sigma(z) \quad(\text{since } L_x(\rho_z) = L_x(\delta_z))\\
=  & \int  h(|x-y|)\, dL_x(\sigma)(y).
\end{align*}
The last thing to check is that the barycenter of $L_x(\sigma)$ is the same as that of $\sigma$, and the $L_x$-symmetry of $\sigma$ is used for this. This was already shown for the simple case and can be shown for the general case by a  similar calculation as above. The detail is as follows:
\begin{align*}      
{\rm bary}(\sigma) &:= \int y\,d\sigma(y) = \frac{1}{2}\bigg{(}  \int z \,d\sigma(z)  + \int \tilde z \,d\sigma(\tilde z)   \bigg{)}\\
&=\int \frac{1}{2}(z + \tilde z) \,d\sigma(z) \quad (\text{since $\sigma$ is $L_x$-symmetric})\\
&= \int {\rm bary}(\rho_z) \,d\sigma(z)\\
&= \int {\rm bary}(L_x(\rho_z)) \,d\sigma(z) \quad (\text{equality was shown for this simple case})\\
&=\iint y \,L_x(\rho_z)(y) \,d\sigma(z) = \int y \, dL_x(\sigma)(y) = {\rm bary} (L_x(\sigma)).
\end{align*}

\end{proof}
Finally, we apply the symmetrization arguments in conjunction with the deformation lemma to prove Theorem \ref{mainresult}. Recall that without loss of generality we can assume $\mu \wedge \nu = 0$.

\begin{proof} Let $\pi$ be a minimizer. Then cost$(\pi)$=cost$(\mathcal{S}\pi)$, so we will assume that $\pi =\mathcal{S}\pi$. Let $d \pi(x,y) = d \pi_x(y) d\mu(x)$. Now we claim that for $\mu$ a.e. $x$, $\pi_x$ is concentrated on the line $L_x$. 

To see this, let $F^+(x,y) = (x, f^+_x(y))$, $F^-(x,y) = (x, f^-_x(y))$, and $w^+ (x,y)= \frac{1}{2}+\frac{1}{2} \frac{x}{|x|}\cdot \frac{y}{|y|}$, $w^- (x,y)= \frac{1}{2}-\frac{1}{2} \frac{x}{|x|}\cdot \frac{y}{|y|}$ for $x, y \neq 0$. Write $\pi = \pi_* + \pi_0$ where $\pi_0 := \pi |_{\R^d \times \{0\}}$. Note that then $ \pi_*(\{0\} \times \R^d) = \pi_*(\R^d \times\{0\} ) = 0$. 

Now define the measure $\rho$ just as in Lemma \ref{variational}:
\begin{align*}
\rho = \pi_0 + {F^+}_\#(w^+\pi_*) + {F^-}_\#(w^-\pi_*).
\end{align*}
 Then the argument in Lemma \ref{variational} along with Lemma \ref{symmetryproperty}(2) we immediately see that cost($\pi$) $>$ cost($\rho$) as soon as $\pi_x$ is not concentrated on $L_x$. Furthermore, by the unchanging nature of barycenter in Lemma \ref{variational} we see that $\rho$ is a martingale transport. However,  note that $\rho$ is not necessarily  in {\rm MT}$(\mu, \nu)$ as it may not have $\nu$ as its second marginal. 

Let $d \rho(x,y) = d\rho_x(y) d\mu(x)$ and let $\tilde \nu$ be the second marginal of $\rho$. Let $A \subset \R^d$ be any annulus set. Then since $\pi_x$ and $\rho_x$ are $R$-equivalent for $\mu$-a.e. $x$,
$$\nu(A) = \int \pi_x(A) d\mu(x) =   \int \rho_x(A) d\mu(x) =\tilde \nu(A) $$
implying that $\mathcal{R} \tilde \nu = \nu$. Then Lemma \ref{symmetryproperty}(1) implies that $\mathcal{S}\rho \in$ {\rm MT}$(\mu, \nu)$, a contradiction to the fact that cost($\pi$) $>$ cost($\rho$) = cost($\mathcal{S}\rho$).  Hence the claim is true for every minimizer $\pi \in $ {\rm MT}$(\mu, \nu)$ (even without assuming $\pi =\mathcal{S}\pi$, since if $\pi$ does not satisfy the claim then neither does $\mathcal{S}\pi$). 

Now the problem \eqref{MGTP} is decomposed to the problem on each one-dimensional subspace with induced marginals as follows: let $\mathcal{L}$ be the set of all one-dimensional subspaces of $\R^d$, and for each $L \in \mathcal{L}$ let $L_0:=L \setminus \{0\}$. Disintegrate $\pi$ along the family of disjoint sets $\{L_0 \times \R^d \,|\, L \in \mathcal{L}\}$ (recall $\mu(\{0\})=0$) and denote it as $(\pi_L)_{L \in \mathcal{L}}$. The above proof shows that in fact $\pi_L$ is concentrated on $L_0 \times L$ hence it is a one-dimensional martingale transport. Then Corollary \ref{coro} provides the uniqueness of each $\pi_L$, which in turn yields the  uniqueness of $\pi$.
\end{proof}

\end{document}